\documentclass[10pt]{amsart}
\usepackage{mystyle}
\usepackage{enumitem}
\usepackage[dvipsnames]{xcolor}
\usepackage{tikz-cd}
\usepackage[colorlinks=true,linkcolor=black,citecolor=blue,urlcolor=blue]{hyperref}

\newenvironment{altenumerate}
   {\begin{list}
      {\textup{(\theenumi)} }
      {\usecounter{enumi}
       \setlength{\labelwidth}{0pt}
       \setlength{\labelsep}{0pt}
       \setlength{\leftmargin}{0pt}
       \setlength{\itemsep}{\the\smallskipamount}
       \renewcommand{\theenumi}{\arabic{enumi}}
      }}
   {\end{list}}

\newcommand{\Bl}{{\rm Bl}}

\newcommand{\GSp}{{\rm GSp}}
\newcommand{\Spec}{{\rm Spec\,}}
\newcommand{\sgn}{{\rm sgn}}
\newcommand{\Proj}{{\rm Proj}}
\newcommand{\bfa}{{\mathbf a}}
\newcommand{\bfb}{{\mathbf b}}

\hypersetup{
  pdftitle={Equivariant regular resolutions of symplectic and split orthogonal local models over the reflex field},
  pdfauthor={Jie Yang, Ioannis Zachos, and Zhihao Zhao},
  pdfsubject={Regular resolutions of local models over the reflex field},
  pdfkeywords={local models, Shimura varieties, resolution of singularities, normal crossings, parahoric level}
}

\newcommand{\Addresses}{{%
  \bigskip
  \footnotesize
  \raggedright

  \textsc{Yau Mathematical Sciences Center, Tsinghua University,
  Beijing, 100084, China}\par\nopagebreak
  \textit{E-mail address:} \texttt{jie-yang@mail.tsinghua.edu.cn}\\

  \textsc{Department of Mathematics, Universit\"at M\"unster,
  M\"unster, 48149, Germany}\par\nopagebreak
  \textit{E-mail address:} \texttt{io.zachos@uni-muenster.de}\\

  \textsc{Department of Mathematics, University of Science and Technology
  Beijing, Beijing, 100083, China}\par\nopagebreak
  \textit{E-mail address:} \texttt{zhihaozhao@ustb.edu.cn}
}}
\title[Regular models of Shimura varieties]
{Regular integral models of Shimura varieties and a conjecture of Pappas}
\author{Jie Yang, Ioannis Zachos, and Zhihao Zhao}
\date{}

\begin{document}

\begin{abstract}
We construct regular integral models for a class of Shimura varieties
of symplectic and orthogonal type with maximal parahoric level at an
odd prime $p$. These models are defined over the ring of integers of
the reflex field and have special fiber with normal crossings and
irreducible components of multiplicity one or two. Our construction
uses explicit equivariant modifications of the corresponding canonical
local models for symplectic and split even orthogonal similitude groups,
with the minuscule cocharacters corresponding to maximal isotropic Grassmannians.
These modifications are obtained by successively blowing up Schubert
varieties in the special fiber. This proves the equivariant modification
conjecture of Pappas at maximal parahoric level in the
above cases.
\end{abstract}

\maketitle

\section{Introduction} 
{Shimura varieties are higher-dimensional generalizations of modular curves and play a central role in number theory. They are quasi-projective smooth varieties canonically defined over number fields called reflex fields. One active area of interest in the study of Shimura varieties is the construction of reasonable integral models over $p$-adic integer rings. Such models are useful in computing the Hasse--Weil zeta function of a Shimura variety in terms of automorphic L-functions, which is part of the Langlands program. For Shimura varieties of abelian type with (quasi-)parahoric level structure, integral models have been constructed in \cite{kisin2010integral,kisin2018integral,KPZ,daniels2026conjecture,MW26}. In \cite{pappas2024p}, Pappas and Rapoport formulated a conjectural theory of \dfn{canonical integral models} with (quasi-)parahoric level structure, which has been verified for Shimura varieties of abelian type; see \cite{daniels2026conjecture,pappas2024p,daniels2025canonical,MW26}. These canonical integral models with general parahoric level structure are often singular, and the construction of regular integral models of Shimura varieties is an important problem in arithmetic
geometry.} 
It also fits into the broader problem of resolution of
singularities in mixed characteristic.

For canonical integral models with parahoric level
structure, the singularities are controlled by the corresponding
canonical local models. These local models{, though originally studied in the context of Shimura varieties, can be constructed for any \dfn{local model triple} $(G,\CG,\cbra{\mu})$} 
where $G$ is a connected
reductive group over a $p$-adic field $F$, $\CG$ is a (quasi-)parahoric
group scheme over $\CO_F$ with generic fiber $G$, and $\{\mu\}$
is a $G(\overline F)$-conjugacy class of minuscule cocharacters
of $G_{\overline F}$. {For the theory of canonical local models, we refer to \cite{scholze2020berkeley, AGLR22, PZ}.} Via the local model diagram, equivariant
resolutions of these local models give rise to regular integral
models of the Shimura varieties. This motivates the following
conjecture of Pappas;
see \cite[\S 2.12]{Pappas} and \cite[\S 8.4]{PRS}.

\begin{conj}\label{conj:equivariant-modification}
Let $(G,\CG,\{\mu\})$ be a local model triple with reflex field $E$,
and let $\RM^{\loc}_{\CG,\mu}$ be its canonical local model over
$\CO_E$. There is a $\CG_{\CO_E}$-equivariant proper birational morphism
$\RM^{\mathrm{reg}}\longrightarrow\RM^\loc_{\CG,\mu}$, which is an isomorphism on the
generic fiber, such that $\RM^{\mathrm{reg}}$ is regular and its special
fiber is a (possibly non-reduced) divisor with simple normal crossings.
\end{conj}

We call a morphism $\RM^{\reg}\ra \RM^\loc_{\CG,\mu}$ satisfying the conditions of Conjecture \ref{conj:equivariant-modification} a \dfn{(regular) equivariant modification of $\RM_{\CG,\mu}^\loc$}. 
We prove Conjecture \ref{conj:equivariant-modification} at \dfn{maximal parahoric level} in the following
cases. Let $F_0$ be a $p$-adic field with ring of integers $\CO$,
uniformizer $\pi_0$, and residue field $k$ of characteristic $p>2$.
Assume that 
\begin{equation*}\tag{$\dagger$}\label{ass:groups}
\begin{aligned}
 &G=\GSp_{2g,F_0},\quad
   \mu=(1^{(g)},0^{(g)});\qquad\text{or}\\
 &{G=\GSO_{2g,F_0}{\text{ is split}},\quad \mu=\mu_+\coloneqq(1^{(g)},0^{(g)})
   \ \text{or}\ 
   \mu_-\coloneqq(1^{(g-1)},0,1,0^{(g-1)}) }.  
\end{aligned}
\end{equation*}
In both cases, the reflex field is $F_0$.

{
\begin{thm}\label{thm:main}
Fix $(G,\{\mu\})$ as in \eqref{ass:groups}.
For any maximal parahoric group scheme $\CG$, let
$\RM\coloneqq \RM^{\loc}_{\CG,\mu}$ be the associated canonical
local model over $\CO$. Let $s$ denote the number of Schubert varieties of positive
codimension in the special fiber $\RM_k$. Then these varieties form a chain
\(
   S_0\subset\cdots\subset S_{s-1}.
\)
Let
\[
\RM=Z_0
\xleftarrow{b_1}
Z_1
\xleftarrow{b_2}
\cdots
\xleftarrow{b_i}
Z_s
\]
be the sequence of blow-ups such that $b_1$ is the blow-up of
$Z_0=\RM$ along $S_0$, and for $2\leq \ell\leq i$, the morphism $b_\ell:Z_\ell\longrightarrow Z_{\ell-1}$ is the blow-up along the strict transform of $S_{\ell-1}$ in $Z_{\ell-1}$. 

Then the resulting scheme $\RM^{\mathrm{reg}}=Z_s$ is regular, and its special fiber
has simple normal crossings with the multiplicity of each irreducible component at most two.
\end{thm}
}

More precisely, $\RM^{\mathrm{reg}}$ is covered by affine open
subschemes isomorphic to
\begin{equation}\label{eq:main-charts}
 \Spec\frac{\CO[x_0,x_1,\ldots,x_N]}
 {\bigl(\pi_0-x_0(x_1\cdots x_s)^2\bigr)},
 \qquad
 N=
 \begin{cases}
   g(g+1)/2,&G=\GSp_{2g},\\
   g(g-1)/2,&G=\GSO_{2g}.
 \end{cases}
\end{equation}

{
\begin{remark}
\begin{altenumerate}
\item Together with earlier results, Theorem~\ref{thm:main} establishes
Conjecture~\ref{conj:equivariant-modification} at maximal parahoric
level in the following cases for classical groups.
For unramified groups of type~$A$, every maximal parahoric subgroup is
hyperspecial, so the canonical local model is smooth and we may
take $\RM^\reg=\RM^\loc_{\CG,\mu}$.
For the orthogonal local models of types~$B$ and~$D$ associated
with the minuscule cocharacter corresponding to the quadric of isotropic lines,
regular equivariant modifications were constructed by Pappas
and the second author in \cite{PaZa}. The symplectic and split even orthogonal cases considered here are
covered by Theorem~\ref{thm:main}, while the corresponding
quasi-split non-split orthogonal cases are covered by
\cite[Proposition~5.3]{YZZss}.
In the latter cases, the construction in \cite[\S~5]{YZZss}
in fact gives semi-stable modifications over the ring of integers of the reflex field. Moreover, for more cases of regular/semi-stable equivarant modifications in the literature, in some cases after a finite extension of the reflex field, we refer the reader to \cite[\S 2, Table 1]{YZZss}. 

\item
In \cite[Theorems~6.1 and~9.1]{YZZss}, we constructed semi-stable
modifications of the split even orthogonal and symplectic local models
after the ramified quadratic extension
$F_0(\sqrt{-\pi_0})/F_0$.
The proof of Theorem~\ref{thm:main} follows the same general strategy,
but working over the original reflex field introduces two additional
difficulties.
First, the inductive chart descriptions in \textit{loc.\ cit.}
depend on the ramified quadratic extension and must be replaced by
an induction argument defined over $\CO_{F_0}$.
{Second, the successive blow-ups in the present construction
have non-reduced special fibers, which makes the plot thicker 
in determining the intermediate blow-up charts; for example, in establishing the
flatness and integrality results of
Proposition~\ref{prop:integrality}}.
\end{altenumerate}
\end{remark}
}

The proof of the above theorem uses the affine charts of the local models and their Schubert stratifications described in \cite{Goertz,YangSplit,YangTopFlat,YZZss}. On these charts, the chain of Schubert varieties are defined by determinantal ideals 
$(\pi_0,\wedge^{r+1}X)$ for $0\leq r<s$, where $X$ is a $2s\times2s$ square matrix and $s$ is as in Theorem~\ref{thm:main}. Let $\mathcal U(d,\ell)$ denote the auxiliary matrix scheme introduced in \S \ref{subsec-matrixsch}, which is integral and $\CO$-flat by Proposition \ref{prop:integrality}. Roughly, the scheme $\mathcal U(d,\ell)$ is defined using a $2d\times 2d$ matrix $X$ satisfying the matrix
relations of the affine local-model charts, with $\pi_0$
replaced by $T_0$, together with the relation
\[
                    \pi_0=T_0(T_1\cdots T_\ell)^2.
\]
Each affine chart of the local model is a product of an
affine space over $\CO$ with a scheme $\mathcal U(s,0)$.
We develop an {inductive} way to show that the blow-up of $\mathcal U(d,\ell)$ along the first determinantal ideal
$(X,T_0)$ is covered by affine charts of the form
\[
             \mathbb A_{\CO}^{b_d}
             \times_{\CO}\mathcal U(d-1,\ell+1),
 \qquad
 b_d=
 \begin{cases}
   4d-2,&\text{symplectic case},\\
   4d-4,&\text{orthogonal case}.
 \end{cases}
\]
Under these identifications, the strict transforms of the
remaining determinantal ideals are products of
$\mathbb A_{\CO}^{b_d}$ with the corresponding determinantal
ideals in $\mathcal U(d-1,\ell+1)$. 
{Iterating this construction,} after $s$ successive blow-ups, the resulting charts are
products of affine spaces with
\[
 \Spec\frac{\CO[T_0,\ldots,T_s]}
 {\bigl(\pi_0-T_0(T_1\cdots T_s)^2\bigr)}.
\]
This gives \eqref{eq:main-charts}, and hence proves Theorem \ref{thm:main}.

Theorem~\ref{thm:main} also has an application to the construction of regular integral models of Shimura varieties. Let $\mathscr S_K$ be the canonical integral model with a local model diagram  (see e.g., \cite[\S 11.1]{YZZss})
\[
\begin{tikzcd}
& \widetilde{\mathscr S}_K
  \arrow[dl, "\pi_K"']
  \arrow[dr, "q_K"] & \\
\mathscr S_K && \RM
\end{tikzcd}
\]
over $\CO$, where $\pi_K$ is a $\CG$-torsor, $q_K$ is smooth
and $\CG$-equivariant, and $\RM$ is as in
Theorem~\ref{thm:main}. For $0\leq j<s$, let $\mathscr Z_j\subset\mathscr S_{K,k}$
be the closure of the Kottwitz--Rapoport stratum corresponding
to the open Schubert cell in $S_j$. Under the local model
diagram, we have
\[
                 \pi_K^{-1}(\mathscr Z_j)=q_K^{-1}(S_j).
\]
As in \cite[Remark 11.2]{YZZss}, after successively
blowing up $\sS_K$ along $\mathscr Z_0,\ldots,\mathscr Z_{s-1}$, we obtain a regular integral model $\sS_K^\reg$, which admits a local
model diagram with $\RM$ replaced by $\RM^{\mathrm{reg}}$.
This is a linear modification in the sense of \cite{Pappas}.
\begin{cor}\label{cor:shimura-regular}
The scheme $\mathscr S_K^{\mathrm{reg}}$ is a regular integral
model with generic fiber equal to $ \mathscr S_K \otimes_{\mathcal{O}} F_0  $. Its special fiber
has normal crossings with the multiplicity of each irreducible component at most two.
\end{cor}

\begin{example}
In the Siegel case, take $F_0=\mathbb Q_p$ and
$\CO=\mathbb Z_p$. Fix $0\leq r\leq g$ and an integer
$N\geq3$ prime to $p$, and put
\[
    J_r=(r+2g\mathbb Z)\cup(-r+2g\mathbb Z).
\]
Let $\mathscr A_{g,J_r,N}$ be the integral Siegel moduli
scheme over $\BZ_p$ with maximal parahoric level at $p$ indexed by
$J_r$ and full level-$N$ structure away from $p$;
see \cite[Definition~4.1]{Hoff}. It parametrizes isomorphism classes of triples $(A,\lambda,\eta)$, where $A$ is an abelian scheme of
relative dimension $g$, $\lambda:A\longrightarrow A^\vee$
is a polarization satisfying
\[
    \ker(\lambda)\subset A[p],
    \qquad \deg(\lambda)=p^{2r},
\]
and $\eta$ is a full level-$N$ structure compatible
with the Weil pairing up to similitude, in the sense
of \textit{loc.\ cit.}

By \cite[Proposition~4.2]{Hoff},
$\mathscr A_{g,J_r,N}$ admits a local model diagram
with the corresponding symplectic local model $\RM$.
Corollary~\ref{cor:shimura-regular} therefore yields
a regular modification $\sA^{\reg}_{g,J,N}$ of $\mathscr A_{g,J_r,N}$
over $\mathbb Z_p$, obtained by successive blow-ups
along closures of Kottwitz--Rapoport strata.
Its special fiber has normal crossings, with each
irreducible component having multiplicity at most two. {It would be very interesting if one can find explicit moduli interpretation for the regular integral model $\sA^{\reg}_{g,J,N}$.}
\end{example}

{\bf Acknowledgements:} 
We thank Y. Luo and G. Pappas for useful comments and suggestions. The author J. Yang was supported by the Shuimu Tsinghua Scholar Program of Tsinghua University. The author I. Zachos was supported by Germany's Excellence Strategy EXC~2044--390685587 ``Mathematics M\"unster: Dynamics--Geometry--Structure" and by the CRC~1442 ``Geometry: Deformations and Rigidity'' of the DFG. The author Z. Zhao was supported by the Tianyuan Fund for Mathematics of NSFC (Grant No. 12526543),  and the Fundamental Research Funds for the Central Universities (Grant No. FRF-TP-26-121).

\medskip
\textbf{AI usage:} The authors used ChatGPT to assist with typos and language editing, and to help identify a sign error in the proof of Proposition \ref{prop:matrix-charts} in the first draft.

\section{ {Local models and their resolutions} }
\label{sec:preliminaries}

We retain assumption \eqref{ass:groups} and the notation $\CO,\pi_0,k$.
Let $H_d$ denote the unit anti-diagonal matrix of size $d$, and put
\[
 B_d^-\coloneqq\begin{pmatrix}0&H_d\\-H_d&0\end{pmatrix},
 \qquad
 B_d^+\coloneqq\begin{pmatrix}0&H_d\\ H_d&0\end{pmatrix}.
\]

{
Let $e_1,\ldots,e_{2g}$ denote the standard basis of $V\coloneqq F_0^{2g}$. Equip $V$ with the bilinear form $B_g^\pm$ under the standard basis.
Every maximal parahoric level is $G(F_0)$-conjugate to the connected
stabilizer of a standard vertex lattice in $V$
\begin{equation}\label{eq:standard-lattice}
 \Lambda_i\coloneqq\CO\langle
 \pi_0^{-1}e_1,\ldots,\pi_0^{-1}e_i,
 e_{i+1},\ldots,e_{2g}\rangle,
\end{equation}
where $0\leq i\leq\lfloor g/2\rfloor$, with $i\neq1$ in the
orthogonal case; see \cite[Proposition~2.11]{YangTopFlat} and
\cite[\S 8]{YZZss}. Write $\CG_i$ for the corresponding parahoric group scheme and abbreviate
}

\[
                         \RM_i\coloneqq\RM^{\loc}_{\CG_i,\mu}.
\]
The cocharacter $\mu$ is as in ($\dagger$). 
The model $\RM_0$ is smooth, so only
$i>0$ requires a blow-up calculation.

\subsection{Auxiliary matrix schemes} \label{subsec-matrixsch}

Let $d\geq1$, $\ell\geq0$, and $Q_\ell\coloneqq T_1\cdots T_\ell$,
with $Q_0=1$. For a $2d\times2d$ matrix of variables $X=(x_{a,b})$, set
\begin{equation*}
 \CX(d,\ell)\coloneqq\Spec
 \frac{\CO[X,T_0,T_1,\ldots,T_\ell]}
 {\bigl(XB_d^-X^t-T_0B_d^-,\ X^tB_d^-X-T_0B_d^-,\
          \pi_0-Q_\ell^2T_0\bigr)}.
\end{equation*}

{Fix $\delta\in\cbra{\pm 1}$.} {For
subsets $U$ and $U'$ of $\{1,\ldots,2d\}$ with the same cardinality, write
$[U:U'](X)$ for the minor of $X$ with its rows and columns given by $U$ and $U'$ in increasing order. Put
\[
 a^*\coloneqq 2d+1-a,\qquad
 U^\perp\coloneqq\{1,\ldots,2d\}\setminus\{a^*:a\in U\}.
\]
For $U\subset\{1,\ldots,2d\}$, let $\sigma_U$ be the permutation of
$\{1,\ldots,2d\}$ sending $\{1,\ldots,|U|\}$ to $U$ in increasing
order and $\{|U|+1,\ldots,2d\}$ to the complement $U^c$ in increasing order. 
Define $\CJ_d^\delta(T_0)$ to be generated by {the following ``spin conditions"}
\begin{align*}
 [U:U'](X)&-\delta\,\sgn(\sigma_U)\sgn(\sigma_{U'})
 [U^\perp:{U'}^{\perp}](X),&&|U|=|U'|=d,  \\
 [U:U'](X)&+\delta\,\sgn(\sigma_U)\sgn(\sigma_{U'})T_0
 [U^\perp:{U'}^\perp](X),&&|U|=|U'|=d+1.
\end{align*}
}
{Define}
\begin{equation*}
 \CY^\delta(d,\ell)\coloneqq\Spec
 \frac{\CO[X,T_0,T_1,\ldots,T_\ell]}
 {\bigl(XB_d^+X^t+T_0B_d^+,\ X^tB_d^+X+T_0B_d^+,\
         \pi_0-Q_\ell^2T_0\bigr)+\CJ_d^\delta(T_0)}.
\end{equation*}

For $d=0$, we set
\begin{equation}\label{eq:terminal-family}
 \CX(0,\ell)=\CY^+(0,\ell)\coloneqq
 \Spec\frac{\CO[T_0,T_1,\ldots,T_\ell]}
 {\bigl(\pi_0-Q_\ell^2T_0\bigr)},
 \qquad \CY^-(0,\ell)\coloneqq\emptyset.
\end{equation}

We will use the following basic properties of the auxiliary schemes
throughout the blow-up calculations in \S \ref{sec:calculation}.

\begin{prop}\label{prop:integrality}
The schemes $\CX(d,\ell)$ and $\CY^\delta(d,\ell)$ (if non-empty) are integral and flat over $\CO$. Their total dimensions are, respectively, $ d(2d+1)+\ell+1$ and $  d(2d-1)+\ell+1.$
\end{prop}

{
\begin{proof}
For $d=0$ this follows from \eqref{eq:terminal-family}. Suppose $d>0$,
and let $R$ be the coordinate ring of either $\CX(d,\ell)$ or $\CY^\delta(d,\ell)$. Write
$B=B_d^\varepsilon$ with $\varepsilon=\pm 1$, set 
\[
 A\coloneqq
 \frac{\CO[X,T_0,T_1,\ldots,T_\ell]}
 {\bigl(XBX^t+\varepsilon T_0B,\ X^tBX+\varepsilon T_0B\bigr)+\CJ},
\]
where $\CJ\coloneqq 0$ for $\CX(d,\ell)$ and
$\CJ\coloneqq \CJ_d^\delta(T_0)$ for $\CY^\delta(d,\ell)$. Then we have \begin{flalign*}
    R = A/(\pi_0-T_0Q_\ell^2).
\end{flalign*}
Let $F\coloneqq F_0(\pi)$ with $\pi^2=-\pi_0$ be a quadratic extension of $F_0$. Then the inclusion \begin{flalign*}
    \CO[T_0]\hookrightarrow \CO_F[T_0,S]/(S^2+T_0)
\end{flalign*}
is a direct summand as $\CO[T_0]$-modules. This induces an injection \begin{flalign*}
    R\hookrightarrow R\otimes_{\CO[T_0]}(\CO_F[T_0,S]/(S^2+T_0)) = \frac{A_1}{(\pi-SQ_\ell)(\pi+SQ_\ell)},
\end{flalign*} 
where $A_1\coloneqq (A\otimes_{\CO}\CO_F)[S]/(S^2+T_0)$.
By \cite[Propositions~6.3 and~9.3]{YZZss}, the rings
$A_1/(\pi-SQ_\ell)$ and $A_1/(\pi+SQ_\ell)$ are integral and
$\CO_F$-flat. Put
\[
 a\coloneqq \pi-SQ_\ell,\qquad b\coloneqq \pi+SQ_\ell.
\]
Then $(a)$ and $(b)$ are prime ideals in $A_1$.
For any $x\in (a)\cap (b)\sset A_1$, write $x=bx_1$ for some $x_1\in A_1$. Then $x=(2\pi-a)x_1\in (a)$ implies that $\pi x_1\in (a)$. Since $A_1/(a)$ is $\CO_F$-flat, we obtain that $x_1\in (a)$, and hence $x\in (ab)$. Therefore
\[
 (a)\cap(b)=(ab),
\]
and consequently
\[
 R\hookrightarrow A_1/(ab) \hookrightarrow A_1/(a)\times A_1/(b).
\]
By $\CO_F$-flatness, $A_1/(a)\times A_1/(b)$ is $\pi$-torsion-free.
Thus $R$ is $\pi_0$-torsion-free, and hence $R$ is flat over $\CO$.

On the generic fiber, the relation $\pi_0=T_0Q_\ell^2$ implies that $T_0,T_1,\ldots,T_\ell$ are invertible. 
Under the transformation $$W\coloneqq  \begin{pmatrix}
 -\varepsilon T_0\inverse I_d&0\\
 0&I_d
 \end{pmatrix}\cdot X, $$ 
we obtain that (cf. the proof of \cite[Lemma 6.4 (2)]{YZZss}) 
\[
R[1/\pi_0]\simeq
\frac{F_0[W,T_0^\pm,\ldots,T_\ell^\pm]}
{(WBW^t-B, W^tBW-B, f, \pi_0-T_0Q_\ell^2)},
\]
where $f=0$ for $\CX(d,\ell)$, and $f=\det(W)-\delta$ for $\CY^\delta(d,\ell)$. Here, we use that in the case $\CY^\delta(d,\ell)$, the Jacobi's identity on minors (\cite[Lemma 4.3 (3)]{YangSplit}) shows that over the generic fiber, the spin conditions in $\CJ^\delta_d(T_0)$ are equivalent to $\det(W)=\delta$.

Hence the generic fiber is isomorphic to
\[
 \Sp_{2d,F_0}\times_{F_0}(\mathbb G_m)^\ell
 \quad\text{or}\quad
 \mathrm O_{2d,F_0}^{\det=\delta}
 \times_{F_0}(\mathbb G_m)^\ell,
\]
respectively. Here $\mathrm O_{2d,F_0}^{\det=\delta}$ denotes the connected component of the orthogonal group $\mathrm O_{2d}$ cut out by $\det(W)=\delta$. In particular, the generic fiber is geometrically integral, of dimension
$d(2d+1)+\ell$ in the symplectic case and $d(2d-1)+\ell$ in the
orthogonal case. Since $R$ is flat over $\CO$, it embeds into $R[1/\pi_0]$. Thus $R$ is integral with the stated total dimension.
\end{proof}
}

For either family $\CX(d,\ell)$ or $\CY^\delta(d,\ell)$, let
\begin{equation}\label{eq:auxiliary-centers}
 C_r(d,\ell)\coloneqq V(T_0,\wedge^rX),\qquad 1\leq r\leq d,
\end{equation}
denote the subscheme cut out by the ideal generated by $T_0$ and the determinantal ideal $\wedge^r X$.
Here $\wedge^rX$ denotes the ideal generated by the $r\times r$ minors of $X$. These subschemes $C_r(d,\ell)$ will be the centers that we are going to blow up in the next section.

\subsection{{Local charts of $\RM_i$ and their resolutions}}

\begin{prop}\label{prop:local-input}
Let $i>0$ be as in \eqref{eq:standard-lattice}. The local model $\RM_i$
is covered by affine open subschemes $\RU_i$ with isomorphisms
\[
 \RU_i\simeq
 \begin{cases}
 \BA_{\CO}^{a_i}\times_{\CO}\CX(i,0),
       &G=\GSp_{2g},\\[2pt]
 \BA_{\CO}^{a_i}\times_{\CO}\CY^\delta(i,0),
       &G=\GSO_{2g}, 
 \end{cases}  
\]
where
\[
 a_i\coloneqq
 \begin{cases}
 (g-2i)(g+2i+1)/2,&G=\GSp_{2g},\\
 (g-2i)(g+2i-1)/2,&G=\GSO_{2g}.
 \end{cases}
\]
There is a chain of Schubert varieties
\[
 S_0\subset S_1\subset\cdots\subset S_{i-1}\subset\RM_{i,k},
\]
stable under $\CG_i$. Under the above identifications, for $0\leq r\leq i-1$, we have
\[
 S_r\cap\RU_i
 =\BA_{k}^{a_i}\times_{k}
   V(\pi_0,\wedge^{r+1}X)
 =\BA_{k}^{a_i}\times_{k}C_{r+1}(i,0),
\]    
{and the scheme $$C_{r+1}(i,0)=\Spec \frac{k[X]}{(XB_i^\pm X^t,X^tB_i^\pm X,\wedge^{r+1}X)}$$ is integral of dimension $r(4i-2r-1)$ (resp. $r(4i-2r+1)$) for $G=\GSO_{2g}$ (resp. $G=\GSp_{2g}$).  }
\end{prop}

{
\begin{proof}
    For $G=\GSp_{2g}$, see \cite[Propositions 8.6, 8.8]{YZZss}. For $G=\GSO_{2g}$, see \cite[Theorem 4.6 (1), Proposition 4.7 (1)]{YZZss}.
\end{proof}
}

\begin{thm} \label{thm:affine}
{
Let $d,\ell\geq0$, and let {$\CZ$} be $\CX(d,\ell)$ or
$\CY^\delta(d,\ell)$. Let $\CZ^\reg$ denote the resulting scheme of the successive blow-ups of $C_1(d,\ell),\ldots,C_d(d,\ell)$. Then $\CZ^\reg$ is covered by covered by affine open subschemes of the form
\begin{equation}\label{eq:affine-resolution}
 \Spec\frac{\CO[x_0,x_1,\ldots,x_N]}
 {\bigl(\pi_0-x_0(x_1\cdots x_s)^2\bigr)},\qquad
 s=d+\ell.
\end{equation}
Here $N=d(2d+1)+\ell$ for $\CX(d,\ell)$ and
$N=d(2d-1)+\ell$ for $\CY^\delta(d,\ell)$.
In particular, it is regular and its special fiber has simple normal
crossings, with component multiplicities at most two.
}
\end{thm}

We will prove Theorem~\ref{thm:affine} in \S \ref{sec:calculation}.
We explain how it implies the main theorem.

\begin{proof}[Proof of Theorem~\ref{thm:main}]
As discussed at the beginning of \S \ref{sec:preliminaries}, we may assume $\RM=\RM_i$ for $0\leq i\leq \lfloor g/2\rfloor$. For $i=0$, the local model is smooth and the assertion is immediate.

Let $i>0$. {Since the
Schubert varieties $S_r$ are $\CG_i$-stable, by construction, the successive blow-up $\RM^\reg\ra \RM$ is a $\CG_i$-equivariant and projective birational morphism with isomorphic generic fibers. }

By Proposition~\ref{prop:local-input}, on each chart $\RU_i$ the blow-up
sequence in Theorem \ref{thm:main} is the product of $\BA_{\CO}^{a_i}$ with the sequence in
Theorem~\ref{thm:affine} for $(d,\ell)=(i,0)$. Theorem~\ref{thm:affine}
therefore gives the affine cover \eqref{eq:main-charts}, and shows that
the resulting scheme is regular with special fiber a simple normal
crossings divisor whose components have multiplicity one or two.
\end{proof}

{
\begin{prop} \label{prop-dimCd}
    Suppose $1\leq r\leq d$. The $k$-scheme $C_r(d,\ell)$ is integral of dimension $(r-1)(4d-2r+3)+\ell$ (resp. $(r-1)(4d-2r+1)+\ell$) for $\CX(d,\ell)$ (resp. $\CY^\delta(d,\ell)$).
\end{prop}
\begin{proof}
    Note that the coordinate ring is \begin{flalign*}
        \frac{k[X,T_1,\ldots,T_\ell]}{(XBX^t, X^tBX,\wedge^r X)}.
    \end{flalign*}
    We have $C_r(d,\ell)\simeq C_r(d,0)\times_k\BA_k^{\ell}$.
    By Proposition \ref{prop:local-input}, we obtain the proposition. 
\end{proof}
}

\section{{Successive blow-ups and the proof of Theorem \ref{thm:affine}} }\label{sec:calculation}

\subsection{The first blow-up}

Fix $d\geq 0$ and $\ell\geq0$, and let $R$ be the coordinate ring of either
$\CX(d,\ell)$ or $\CY^\delta(d,\ell)$. Put $I\coloneqq (X,T_0)$. The
surjection of graded $R$-algebras
\[
 R[(y_{a,b})_{1\leq a,b\leq 2d},\alpha]\longrightarrow
 \bigoplus_{r\geq0}I^r,\qquad
 Y=(y_{a,b})\longmapsto X=(x_{a,b}),\quad\alpha\longmapsto T_0,
\]
where $x_{a,b}$, $T_0\in I$ have degree one, realizes the blow-up along $I$
as a closed subscheme
\[
 \Bl_I(\Spec R)=\Proj\Bigl(\bigoplus_{r\geq0}I^r\Bigr)
 \ \lhook\joinrel\longrightarrow\ \Proj R[y_{a,b},\alpha].
\]
The blow-up $\Bl_I(\Spec R)$ has two types of standard affine charts: the \textit{scalar chart} $D_+(\alpha)$ and
the \textit{matrix charts} $D_+(y_{a,b})$. {We first show that the matrix charts cover the blow-up.} 
{
Note that we have (cf. \cite[0G8S]{stacks-project}) 
\begin{equation}\label{eqDa}
\begin{aligned}
    D_+(\alpha)
    &= \Spec R [\frac{I}{T_0}] \\
    &\simeq \Spec \frac{R[Y,\alpha]}
    {(\alpha-1,\,
    f\mid \text{$\exists n\geq 0,\ T_0^n f\in
    (x_{i,j}-y_{i,j}T_0)\text{ for }1\leq i,j\leq 2d$})}.
\end{aligned}
\end{equation}
and 
\begin{equation} \label{eqY}
\begin{aligned}
    D_+(y_{a,b})
    &= \Spec R [\frac{I}{x_{a,b}}] \\
    &\simeq
    \Spec \frac{R[Y,\alpha]}
    {\left(
        y_{a,b}-1,\,
        f\ \middle|\
        \begin{gathered}
            \exists\, n\geq 0,\;
            x_{a,b}^n f\in
            (x_{i,j}-y_{i,j}x_{a,b},\,T_0-\alpha x_{a,b}),\\
            \text{for }1\leq i,j\leq 2d
        \end{gathered}
    \right)}.
\end{aligned}
\end{equation}

Over the scalar chart $D_+(\alpha)$, by \eqref{eqDa} we have the relation \begin{flalign*}
    T_0YB_d^\varepsilon Y^t=-\varepsilon B_d^\varepsilon.
\end{flalign*}
In particular, $T_0$ and $Y$ is invertible. Thus, locally, at least one $y_{a,b}$ is invertible. We deduce the following.  
\begin{lem}\label{lem-coverscaler}
    The scalar chart $D_+(\alpha)$ is covered by matrix charts $D_+(y_{a,b})$.
\end{lem}
}


For $1\leq a\leq2d$, put
\[
 \tau_a\coloneqq\begin{cases}1,&a\leq d,\\-1,&a>d.\end{cases}
\]

\begin{prop}\label{prop:matrix-charts}
For $1\leq a,b\leq2d$, the matrix charts are
\begin{align}
 D_+(y_{a,b})&\simeq
 \CX(d-1,\ell+1)\times_{\CO}\BA_{\CO}^{4d-2}
 &&\text{for }\CX(d,\ell),\label{eq:symp-matrix-chart}\\
 D_+(y_{a,b})&\simeq
 \CY^{\delta\tau_a\tau_b}(d-1,\ell+1)
 \times_{\CO}\BA_{\CO}^{4d-4}
 &&\text{for }\CY^\delta(d,\ell).\label{eq:orth-matrix-chart}
\end{align}
\end{prop}

\begin{proof} 
We first move the chosen entry to position $(1,1)$. Let $P_a$ and
$P_b$ be the permutation matrices reordering the row and column
indices as $(a,\mathcal{I},a^*)$ and $(b,\mathcal{J},b^*)$, respectively, where
\[
 \mathcal{I}=\{1,\ldots,2d\}\setminus\{a,a^*\},\qquad
 \mathcal{J}=\{1,\ldots,2d\}\setminus\{b,b^*\}
\]
are in increasing order. Then
$\det(P_a)=\tau_a$ and $\det(P_b)=\tau_b$.

In the orthogonal case $\CY^\delta(d,\ell)$, write $D_+^\delta(y_{a,b})$ for the matrix chart $D_+(y_{a,b})$ in $\CY^\delta(d,\ell)$. Note that $P_a$ and $P_b$ preserve $B_d^+$ (that is, $MB_d^+M^t=B_d^+$ for $M=P_a$ or $P_b$). Then one can directly check that the transformation sending $X$ to $P_aXP_b^t$ induces an isomorphism $D_+^\delta(y_{a,b})=D_+^{\delta\tau_a\tau_b}(y_{1,1})$.
In the symplectic case $\CX(d,\ell)$, put
\[
 S_a\coloneqq \operatorname{diag}(1,\ldots,1,\tau_a)P_a,\qquad
 S_b\coloneqq \operatorname{diag}(1,\ldots,1,\tau_b)P_b.
\]
These matrices preserve $B_d^-$, and we replace $X$ by $S_aXS_b^t$. Then we have an isomorphism $D_+(y_{a,b})\simeq D_+(y_{1,1})$. 

Hence, we may reduce to studying the normalized chart $D_+(y_{1,1})$.
By \eqref{eqY}, over $D_+(y_{1,1})$, we have that
\begin{equation*}
 y_{1,1}=1,\qquad X=x_{1,1}Y,\qquad T_0=x_{1,1}\alpha,\qquad
 Y=\begin{pmatrix}
       1&\mathbf a^t&u\\
       \mathbf b&C&\mathbf b^*\\
       v&(\mathbf a^*)^t&c
    \end{pmatrix},
\end{equation*}
where the vectors
\[
\bfa:=(y_{1,2} \dots   y_{1,2d-1})^t,\quad
\bfa^*:=(y_{2d,2} \dots  y_{2d,2d-1})^t,
\]
and
\[
\bfb:=(y_{2,1} \dots y_{2d-1,1})^t,\quad
\bfb^*:=(y_{2,2d} \dots  y_{2d-1,2d})^t,
\]
the middle block $C$ has size $2d-2$. In particular, we obtain

\begin{equation}\label{eq:pivot-form-equations}
 x_{1,1}YB_d^\varepsilon Y^t=-\varepsilon\alpha B_d^\varepsilon,
 \qquad
 x_{1,1}Y^tB_d^\varepsilon Y=-\varepsilon\alpha B_d^\varepsilon,
\end{equation}
together with
\begin{equation}\label{eq:base-pivot}
    \pi_0= x_{1,1}\alpha Q_\ell^2.
\end{equation}
Set $K\coloneqq B_{d-1}^\varepsilon$ and define
\begin{equation}\label{eq:schur-complement}
    Z\coloneqq C-\mathbf b\mathbf a^t.
\end{equation}

\smallskip\noindent
\emph{Alternating form.} 

{In this case, we have $K=B_{d-1}^-$}. For $2\leq j\leq 2d-1$, the $(1,j)$-entries of the
second and first equations in \eqref{eq:pivot-form-equations},
respectively give
\[
 \mathbf a^*=v\mathbf a+C^tK\mathbf b,
 \qquad
 \mathbf b^*=u\mathbf b+CK\mathbf a.
\]
These identities eliminate the starred vectors.
The $(1,2d)$-entry of the first equation in
\eqref{eq:pivot-form-equations} is
\[
 x_{1,1}\bigl(c-uv+\mathbf a^tK\mathbf a^*\bigr)=\alpha.
\]
Substituting the expression for $\mathbf a^*$ and using
$\mathbf a^tK\mathbf a=0$, we obtain $\alpha=x_{1,1}\eta$, where
\begin{equation*}\label{eq:symp-eta}
 \eta\coloneqq c-uv+\mathbf a^tKC^tK\mathbf b.
\end{equation*}
The remaining matrix equations, after saturation, are
\begin{equation}\label{eq:symp-reduced}
 ZKZ^t=\eta K,
 \qquad Z^tKZ=\eta K.
\end{equation}
Indeed, after substituting the expressions for the starred
vectors and $\alpha=x_{1,1}\eta$, the middle block of the first
equation in \eqref{eq:pivot-form-equations} gives, after saturation,
\[
 \eta K
 =CKC^t-CK\mathbf a\mathbf b^t-\mathbf b\mathbf a^tKC^t
 =(C-\mathbf b\mathbf a^t)K(C^t-\mathbf a\mathbf b^t)
 =ZKZ^t,
\]
where we use $\mathbf a^tK\mathbf a=0$.
The middle block of the second equation similarly gives
$Z^tKZ=\eta K$, and the other entries impose no additional relations. The variables $\mathbf a,\mathbf b,u,v$ are free, in number $4d-2$.
After substituting $\alpha=x_{1,1}\eta$, equations
\eqref{eq:base-pivot} and \eqref{eq:symp-reduced} are precisely
the equations for $\CX(d-1,\ell+1)$, with $T_{\ell+1}=x_{1,1}$ and $T_0'=\eta$.

\smallskip\noindent
\emph{Symmetric form.} 

{In this case, we have $K=B_{d-1}^+=H_{2d-2}$}. The $(1,1)$-entries of the first and second equations in
\eqref{eq:pivot-form-equations}, respectively then give
\[
 2u+\mathbf a^tK\mathbf a=0,
 \qquad
 2v+\mathbf b^tK\mathbf b=0.
\]
Since $2$ is invertible, we obtain
\begin{equation*}
 u=-\tfrac12\mathbf a^tK\mathbf a,
 \qquad
 v=-\tfrac12\mathbf b^tK\mathbf b.
\end{equation*}
For $2\leq j\leq 2d-1$, the $(1,j)$-entries of the
second and first equations in \eqref{eq:pivot-form-equations},
respectively then give
\[
 \mathbf a^*=-C^tK\mathbf b
              +\tfrac12(\mathbf b^tK\mathbf b)\mathbf a, \quad \mathbf b^*=-CK\mathbf a
              +\tfrac12(\mathbf a^tK\mathbf a)\mathbf b.
\]
These identities eliminate the starred vectors.
The $(1,2d)$-entry of the first equation in
\eqref{eq:pivot-form-equations} is
\[
 x_{1,1}\bigl(c+uv+\mathbf a^tK\mathbf a^*\bigr)=-\alpha.
\]
Substituting the preceding expressions, we obtain
$\alpha=x_{1,1}\eta$, where
\begin{equation*}
 \eta\coloneqq -c-\tfrac34(\mathbf a^tK\mathbf a)(\mathbf b^tK\mathbf b)
       +\mathbf a^tKC^tK\mathbf b.
\end{equation*}

With \eqref{eq:schur-complement}, the remaining form equations become
\begin{equation*}
    ZKZ^t+\eta K=0,
    \qquad Z^tKZ+\eta K=0. 
\end{equation*}
Indeed, after substituting the expressions for the starred
vectors and $\alpha=x_{1,1}\eta$, the middle block of the first
equation in \eqref{eq:pivot-form-equations} gives, after saturation,
\[
 \begin{aligned}
 -\eta K
 &=CKC^t-CK\mathbf a\mathbf b^t-\mathbf b\mathbf a^tKC^t
   +(\mathbf a^tK\mathbf a)\mathbf b\mathbf b^t\\
 &=(C-\mathbf b\mathbf a^t)K(C^t-\mathbf a\mathbf b^t)
 =ZKZ^t.
 \end{aligned}
\]
The middle block of the second equation similarly gives
$Z^tKZ=-\eta K$, and the other entries impose no additional
form relations. The free variables are now just $\mathbf a,\mathbf b$, in number $4d-4$.

{For the spin condition, 
let \(U,U'\subset\{2,\ldots,2d-1\}\) have the same cardinality, and put
\[
V:=U\cup\{1\},\qquad V':=U'\cup\{1\}.
\]
When \(U,U'\) occur in minors and signatures of \(Z\), we identify
\(\{2,\ldots,2d-1\}\) with \(\{1,\ldots,2d-2\}\) in increasing order
and retain the same notation. With \(U^\perp,(U')^\perp\) computed in
this remaining index set, we have
\[
V^\perp=U^\perp\cup\{1\},\qquad
(V')^\perp=(U')^\perp\cup\{1\}.
\]

By Sylvester's determinant identity {(see e.g., \cite[Formula (2), pp 1572]{salem2008simple})} and \(y_{1,1}=1\), we obtain
\begin{flalign}\label{eq:minor-yz}
    [V:V'](Y)=[U:U'](Z),\qquad [V^\perp:(V')^\perp](Y)
=[U^\perp:(U')^\perp](Z).
\end{flalign}
Moreover, using the increasing reindexing of $U$ and $U'$, the sign terms remain the same, i.e., 
\[
\operatorname{sgn}(\sigma_V)
=\operatorname{sgn}(\sigma_U).
\]
The same equality holds for
\(V'\) and \(U'\). Then the relations in \(J_d^{\delta'}(T_0)\) imply
\[
[U:U'](Z)-\delta'\operatorname{sgn}(\sigma_U)
\operatorname{sgn}(\sigma_{U'})
[U^\perp:{U'}^\perp](Z)=0,
\qquad \text{for}~ |U|=|U'|=d-1,
\]
and
\[
[U:U'](Z)
+\delta'\operatorname{sgn}(\sigma_U)
\operatorname{sgn}(\sigma_{U'})\eta
[U^\perp:{U'}^\perp](Z)=0,
\qquad \text{for}~ |U|=|U'|=d.
\]
These are precisely the relations in \(J_{d-1}^{\delta'}(\eta)\).

Together with the preceding form equations and elimination formulas,
the calculations give closed immersions
\[
D_+(y_{1,1})
\hookrightarrow
\CX(d-1,\ell+1)\times_{\CO}\BA_{\CO}^{4d-2}
\]
in the alternating case, and
\[
D_+(y_{1,1})
\hookrightarrow
\CY^{\delta'}(d-1,\ell+1)
\times_{\CO}\BA_{\CO}^{4d-4}
\]
in the symmetric case, where
\[
T_{\ell+1}=x_{1,1},\qquad T_0'=\eta.
\]

The target rings are integral by Proposition~\ref{prop:integrality}.
Their dimensions are, respectively,
\[
d(2d+1)+\ell+1
\qquad\text{and}\qquad
d(2d-1)+\ell+1.
\]
By the property of blow-up, the chart $D_+(y_{1,1})$ is also integral
and has the same dimension as \(R\). Therefore, both closed immersions are isomorphisms. This proves the proposition.
}

\end{proof}


\subsection{Proof of Theorem~\ref{thm:affine}}
We first identify the strict transforms of the remaining blow-up centers on the matrix charts.
\begin{lem}\label{lem:strict-transforms}
Suppose $1\leq r<d$. Under the isomorphisms of Proposition~\ref{prop:matrix-charts},
the strict transform of $C_{r+1}(d,\ell)$ on each matrix chart
is identified with
\[
 \BA_{k}^{a_d}\times_{k}C_r(d-1,\ell+1),
\]
where $a_d=4d-2$ in the symplectic case and $a_d=4d-4$
in the orthogonal case.
\end{lem}

{
\begin{proof}
    Similarly as in \eqref{eq:minor-yz}, the Sylvester's determinant identity implies that any $r\times r$ minor of $Z$ can be expressed by a $(r+1)\times (r+1)$ minor of $Y$. Hence, the strict transform of $C_{r+1}(d,\ell)=V(T_0,\wedge^{r+1}X)$ on the chart $D_+(y_{a,b})$ contains  $\BA_{k}^{a_d}\times_{k}C_r(d-1,\ell+1)$. Since they are integral of the same dimension by Proposition \ref{prop-dimCd}, the lemma follows. 
\end{proof}
}

We now prove Theorem~\ref{thm:affine} by induction on $d$.
The case $d=0$ follows from \eqref{eq:terminal-family}.
Suppose $d>0$. {By Lemma \ref{lem-coverscaler}, the first blow-up of $\CZ$ along $C_1(d,\ell)$ is covered by matrix charts $D_+(y_{a,b})$.} On each matrix chart, Proposition~\ref{prop:matrix-charts} and Lemma~\ref{lem:strict-transforms} identify the remaining blow-up sequence with the corresponding sequence for $(d-1,\ell+1)$, times an affine space. The induction hypothesis therefore gives a covering of $\CZ^{\reg}$ by charts of the form \eqref{eq:affine-resolution}, {and Theorem \ref{thm:affine} follows. }

\Addresses

\end{document}